\documentclass[12pt,reqno]{amsart}
\usepackage{amssymb,amsmath,tabularx,mathrsfs,mathbbol,yfonts,upgreek}
\usepackage{amsthm,verbatim,comment}
\usepackage[bookmarks=true]{hyperref}
\usepackage{pstricks,pst-node,pst-plot}
\usepackage{geometry}
\usepackage{stmaryrd}
\usepackage{paralist}
\usepackage[all]{xy}
\DeclareSymbolFontAlphabet{\mathbb}{AMSb}
\DeclareSymbolFontAlphabet{\mathbbl}{bbold}

\newtheorem{thm}{Theorem}[section]
\newtheorem{lem}[thm]{Lemma}

\newtheorem{cor}[thm]{Corollary}

\theoremstyle{definition}

\newtheorem{nota}[thm]{Notation}
\newtheorem{eg}[thm]{Example}
\newtheorem{rem}[thm]{Remark}
\theoremstyle{remarks}
\newtheorem*{rem*}{Remarks}

\newcommand{\sgn}{\mathrm{sgn}}

\newcommand{\sym}[1]{\mathfrak{S}_{#1}}
\newcommand{\tsym}[2]{\mathfrak{S}_{#1}^{+#2}}
\newcommand{\Ind}{\operatorname{\mathrm{Ind}}}

\newcommand{\R}{\mathscr{R}}
\newcommand{\Y}{\mathscr{Y}}
\newcommand{\C}{\mathscr{C}}
\newcommand{\NN}{\mathbb{N}}
\renewcommand{\O}{\mathcal{O}}

\newcommand{\e}{\varepsilon}

\newcommand{\Z}{\mathbb{Z}}

\renewcommand{\P}{\mathscr{P}}

\newcommand{\seq}[1]{\mathbf{#1}}
\newcommand{\Sch}{\mathrm{S}}
\renewcommand{\L}{\mathscr{L}}
\newcommand{\List}{\mathrm{List}}
\newcommand{\m}{{\pmb{\mathscr{M}}}}

\newcommand{\set}[1]{[#1]}
\newcommand{\cont}{\text{{\tiny $\#$}}}
\newcommand{\len}[1]{\wp(#1)}
\newcommand{\power}[1]{\llbracket #1\rrbracket}
\newcommand{\B}{\mathcal{B}}
\newcommand{\CC}{\mathcal{C}}
\newcommand{\D}{\mathcal{D}}

\numberwithin{equation}{section}

\begin{document}
\title[Straightening rule]{Straightening rule for an $m'$-truncated polynomial ring}

\author{Kay Jin Lim}
\address[K. J. Lim]{Division of Mathematical Sciences, Nanyang Technological University, SPMS-PAP-03-01, 21 Nanyang Link, Singapore 637371.}
\email[K. J. Lim]{limkj@ntu.edu.sg}

\begin{abstract} We consider a certain quotient of a polynomial ring categorified by both the isomorphic Green rings of the symmetric groups and Schur algebras generated by the signed Young permutation modules and mixed powers respectively. They have bases parametrised by pairs of partitions whose second partitions are multiples of the odd prime $p$ the characteristic of the underlying field. We provide an explicit formula rewriting a signed Young permutation module (respectively, mixed power) in terms of signed Young permutation modules (respectively, mixed powers) labelled by those pairs of partitions. As a result, for each partition $\lambda$, we discovered the number of compositions $\delta$ such that $\delta$ can be rearranged to $\lambda$ and whose partial sums of $\delta$ are not divisible by $p$.
\end{abstract}	

\thanks{Supported by Singapore MOE Tier 2 AcRF MOE2015-T2-2-003.}

\maketitle

\section{Introduction}

The ring of symmetric functions plays an important role in both combinatorics and representation theory of the symmetric groups (see \cite[\S I.7]{Mac} and \cite{Knutson}). In the characteristic zero case, the characteristic map gives an isometric isomorphism between the ring of symmetric functions and the Green ring for the symmetric groups where the Schur functions correspond to the Specht modules, the complete and elementary symmetric functions correspond to the trivial and signature representations for the symmetric groups respectively such that multiplication of two symmetric functions corresponds to induction of the outer tensor product of the two respective modules. In the positive characteristic case, the ring of symmetric functions is isomorphic to the Green ring of the symmetric groups generated by Young permutation modules. Let $\Lambda(X),\Lambda(Y)$ be two rings of symmetric functions in the sets of independent countably infinite commuting variables $X,Y$ respectively with coefficients in $\Z$. Fix a positive integer $m$. In this paper, we study a quotient $\Gamma^{(m)}$, depending on $m$, of the ring generated by $\Lambda(X)$ and $\Lambda(Y)$.

Let $k$ be a field of odd characteristic $p$. In \cite{Donkin}, Donkin showed that the isomorphism classes of indecomposable summands of the signed Young permutation modules \[M(\alpha|\beta)=\Ind^{\sym{n}}_{\sym{\alpha}\times\sym{\beta}}(k(\alpha)\boxtimes \sgn(\beta)),\] where $(\alpha|\beta)$ is a pair of compositions such that the sizes of $\alpha$ and $\beta$ sum up to $n$ and $k(\alpha),\sgn(\beta)$ are the trivial and sign $k\sym{\alpha}$- and $k\sym{\beta}$-modules respectively, are parametrised by the set $\P^2_p(n)$ consisting of pairs of partitions of the form $(\lambda|p\mu)$ such that $|\lambda|+p|\mu|=n$. Furthermore, $M(\lambda|p\mu)$ has a distinguished indecomposable summand the signed Young module $Y(\lambda|p\mu)$ with multiplicity one such that any other summand of $M(\lambda|p\mu)$ is isomorphic to $Y(\delta|p\xi)$ for some $(\delta|p\xi)\rhd(\lambda|p\mu)$, here, $\rhd$ is certain dominance order on $\P^2_p(n)$ (see Subsection \ref{SS: 2.3} below). In other words, the Green ring of the symmetric groups generated by signed Young permutation modules, denoted by $\Y(\sym{})$, has $\Z$-bases $\CC=\{[M(\lambda|p\mu)]:(\lambda|p\mu)\in\P^2_p\}$ and $\{[Y(\lambda|p\mu)]:(\lambda|p\mu)\in\P^2_p\}$ where $\P^2_p=\bigcup_{n\in\NN_0}\P^2_p(n)$.

Let $E$ be the natural module for the Schur algebra $S(\infty,1)$. For each pair of compositions $(\alpha|\beta)$, the mixed power is defined as \[K^{\alpha|\beta}E=S^\alpha E\otimes \bigwedge\nolimits^\beta E,\] where $S^\alpha E=S^{\alpha_1}E\otimes\cdots\otimes S^{\alpha_k}E$, $\bigwedge\nolimits^\beta E=\bigwedge\nolimits^{\beta_1}E\otimes\cdots\otimes \bigwedge\nolimits^{\beta_\ell}E$ if $\alpha=(\alpha_1,\ldots,\alpha_k)$ and $\beta=(\beta_1,\ldots,\beta_\ell)$, and $S^rE$ and $\bigwedge\nolimits^r E$ are the $r$th symmetric and exterior powers of $E$ respectively. For each $(\lambda|p\mu)\in \P^2_p$, the mixed power $K^{\lambda|p\mu}E$ has a distinguished indecomposable summand the listing module $\List^{\lambda|p\mu}E$ of multiplicity one such that any other summand of $K^{\lambda|p\mu}E$ is isomorphic to $\List^{\delta|p\xi}E$ for some $(\delta|p\xi)\rhd(\lambda|p\mu)$. Furthermore, any indecomposable summand of $K^{\alpha|\beta}E$ is isomorphic to a listing module. In other words, the Green ring of the Schur algebras generated by mixed powers, denoted by $\L(\Sch)$, has $\Z$-bases $\D=\{[K^{\lambda|p\mu}E]:(\lambda|p\mu)\in\P^2_p\}$ and $\{[\List^{\lambda|p\mu}E]:(\lambda|p\mu)\in\P^2_p\}$. In fact, $\L(\Sch)$ is isomorphic to $\Y(\sym{})$ induced by the Schur functor where $K^{\alpha|\beta}E$ and $\List^{\lambda|p\mu}E$ are mapped to $M(\alpha|\beta)$ and $Y(\lambda|p\mu)$ respectively.

In the classical case, the class of non-isomorphic direct summands of Young permutation $k\sym{n}$-modules are known as Young modules and are labelled by the set of partitions of $n$. In the positive characteristic case, determining the multiplicity of a Young module as a direct summand of a Young permutation module is an open problem (see \cite{Donkin1, Erdmann, JGrab, James, Klyachko}). The numbers are known as the $p$-Kostka numbers. The signed $p$-Kostka numbers are generalisation of the $p$-Kostka numbers which are defined as the multiplicities of signed Young modules as direct summands of signed Young permutation modules (see \cite{GLOW}). Considering writing a signed Young permutation module as $\Z$-linear combination of the signed Young modules in the Green ring $\Y(\sym{})$ is an open problem, in this article, we present an explicit formula writing a signed Young permutation module (respectively, a mixed power) in terms of the basis $\CC$ (respectively, $\D$). The proof relies on the categorification theorem of the quotient ring $\Gamma^{(p)}$ by both $\Y(\sym{})$ and $\L(\Sch)$, proved by Donkin in \cite{Donkin}. Along the way, for each partition $\lambda$, we discovered the number of compositions $\delta$ such that $\delta$ can be rearranged to $\lambda$ and whose partial sums of $\delta$ are not divisible by $p$. These numbers appear, up to signs, as the coefficients in the sum of the sign representations (or the exterior powers) in terms of the basis $\CC$ (or $\D$).

The article is organised as follows. In the next section, we recall the definitions of symmetric functions, the quotient ring $\Gamma^{(m)}$, the Green rings $\Y(\sym{})$ and $\L(\Sch)$. In Section~\ref{S: m'truncation}, we study some properties of $\Gamma^{(m)}$. In Section~\ref{S: straightening rule}, we prove the main result Theorem~\ref{T: straight sym m} and deduce the formulae for writing a signed Young permutation module and a mixed power in terms of the bases $\CC$ and $\D$ respectively. As a consequence of the formulae, for any pair of compositions $(\alpha|\beta)$, we deduce the `canonical' summand of the signed Young permutation module $M(\alpha|\beta)$ (respectively, the mixed power $K^{\alpha|\beta}E$) in Corollary \ref{C: Young distinguished summand}.

\subsection*{Acknowledgement} The author thank the anonymous referees for their comments and suggestions. 

\section{Preliminaries}\label{S: Prelim}

In this section, we fix the notation we shall require throughout in this article and introduce the background material. The standard references are \cite{Donkin, JGreen, GJAK, Mac}.

Let $\Z$ be the set of integers, let $\NN$ be the set of positive integers and let $\NN_0$ be the set of non-negative integers. For a finite set $A$, let $\sym{A}$ denote the symmetric group acting on the set $A$. For each $n\in \NN_0$, let \[\set{n}=\{i\in\NN:1\leq i\leq n\}\] and $\sym{n}=\sym{[n]}$. By convention, $\set{0}=\emptyset$ and $\sym{0}$ is the trivial group.

Let $n\in\NN_0$. A composition $\lambda$ of $n$ is a sequence of positive integers $(\lambda_1,\ldots,\lambda_r)$ such that $\sum^r_{i=1}\lambda_i=n$. In this case, we write $\ell(\lambda)=r$ and $|\lambda|=n$. By convention, the unique composition of $0$ is denoted as $\varnothing$ and $\ell(\varnothing)=0$. The set of all compositions of $n$ is denoted by $\C(n)$ and we write $\C=\bigcup_{n\in\NN_0}\C(n)$. The composition $\lambda$ is called a partition if $\lambda_1\geq\cdots\geq\lambda_r$. We write $\P(n)$ for the set of partitions of $n$ and $\P=\bigcup_{n\in\NN_0}\P(n)$ for the set of all partitions. If the parts of a composition $\delta$ can be rearranged to a partition $\lambda$, we say that $\delta$ has type $\lambda$. In this case, $\lambda$ is uniquely determined by $\delta$.

The concatenation of two compositions $\alpha$ and $\beta$ is the composition \[\alpha\cont\beta=(\alpha_1,\ldots,\alpha_{\ell(\alpha)},\beta_1,\ldots,\beta_{\ell(\beta)}).\] By convention, $\alpha\cont\varnothing=\alpha$ and, similarly, $\varnothing\cont\beta=\beta$. Let $m\in\NN$ and $\mu=(\mu_1,\ldots,\mu_s)$ be a composition. We write $m\mu$ for the composition $(m\mu_1,\ldots,m\mu_s)$ of $m|\mu|$.

The set of all pairs $(\alpha|\beta)$ of compositions of $n$, i.e., $\alpha,\beta$ are compositions and $|\alpha|+|\beta|=n$, is denoted by $\C^2(n)$. We write $\C^2=\bigcup_{n\in\NN_0}\C^2(n)$. Similarly, we write $\P^2(n)$ for the set of all pairs of partitions of $n$ and $\P^2=\bigcup_{n\in\NN_0}\P^2(n)$. For $m\in\NN$, the subset of $\P^2(n)$ consisting of pairs of the form $(\lambda|m\mu)$ is denoted by $\P^2_m(n)$. Furthermore, we write $\P^2_m=\bigcup_{n\in\NN_0}\P^2_m(n)$.

Throughout, $k$ is a field of odd characteristic $p$.

\subsection{Symmetric functions}\label{S: symm func} Let $\Z\power{X}$ be the set of formal power series in the set consisting of countably infinite commuting variables $X=\{x_i:i\in\NN\}$ with coefficients in $\Z$ such that $\deg(x_i)=1$ for all $i\in \NN$. For each $n\in\NN$, the symmetric group $\sym{n}$ acts on $\Z\power{X}$ by permuting the variables $x_1,\ldots,x_n$ naturally, i.e., for each $\sigma\in\sym{n}$ and $f(X)\in\Z\power{X}$, the element $\sigma\cdot f(X)$ is obtained from $f(X)$ by replacing each $x_i$ with $x_{\sigma(i)}$. The element $f(X)$ is called a symmetric function if it is invariant under the actions $\sym{n}$ for all $n\in\NN$.

The set of symmetric functions with degrees bounded above in the set of variables $X$ is denoted by $\Lambda(X)$. Notice that $\Lambda(X)$ is a commutative graded ring where $\Lambda(X)=\bigoplus_{n\in\NN_0}\Lambda^n(X)$ and $\Lambda^n(X)$ consists of homogeneous symmetric functions of degree $n$. In the theory of symmetric function, $\Lambda(X)$ has different $\Z$-bases of special interest. In particular, we are interested in the elementary and complete symmetric functions which we shall now describe.

For any $n\in\NN$, the $n$th elementary symmetric function and the $n$th complete symmetric function are
\begin{align*}
  e_n(X)&=\sum_{i_1< \cdots< i_n}x_{i_1}\cdots x_{i_n},\\
  h_n(X)&=\sum_{i_1\leq \cdots\leq i_n}x_{i_1}\cdots x_{i_n},
\end{align*} respectively. By convention, let $e_0(X)=1=h_0(X)$ and $h_n(X)=0=e_n(X)$ if $n<0$. For each composition $\alpha=(\alpha_1,\ldots,\alpha_r)$, let
\begin{align*}
e_\alpha(X)&=e_{\alpha_1}(X)\cdots e_{\alpha_r}(X),\\
h_\alpha(X)&=h_{\alpha_1}(X)\cdots h_{\alpha_r}(X).
\end{align*} By convention, $e_\varnothing(X)=1=h_\varnothing(X)$. Clearly, if $\alpha$ has type $\lambda$ then $e_\alpha(X)=e_\lambda(X)$ and $h_\alpha(X)=h_\lambda(X)$. The following results are well-known.

\begin{thm}[{\cite[(2.4, $2.6^\prime$, 2.8)]{Mac}}]\label{T: symm function}\
\begin{enumerate}
  \item [(i)] The ring of symmetric functions $\Lambda(X)$ is the polynomial ring in $\{h_n(X):n\in\NN\}$ and $\{e_n(X):n\in \NN\}$ respectively over $\Z$. Furthermore, the subsets $\{e_\lambda(X):\lambda\in\P\}$ and $\{h_\lambda(X):\lambda\in \P\}$ are $\Z$-bases for $\Lambda(X)$.
  \item [(ii)] For each $d\in\NN$, we have $\sum_{i=0}^d(-1)^ih_i(X)e_{d-i}(X)=0$.
\end{enumerate}
\end{thm}

\subsection{An $m'$-truncated polynomial ring $\Gamma^{(m)}$} Let $X=\{x_i:i\in \NN\}$ and $Y=\{y_j:j\in\NN\}$ be two sets of independent countably infinite commuting variables such that $\deg(x_i)=1=\deg(y_i)$ for all $i\in\NN$. Clearly, $\Lambda(X)$ and $\Lambda(Y)$ are graded subrings of $\Z\power{X\cup Y}$. Let $\langle \Lambda(X),\Lambda(Y)\rangle $ be the subring of $\Z\power{X\cup Y}$ generated by $\Lambda(X)$ and $\Lambda(Y)$. By Theorem \ref{T: symm function}(i), it is clear that $\langle \Lambda(X),\Lambda(Y)\rangle$ is a polynomial ring in $\{h_i(X),e_j(Y):i,j\in\NN\}$ over $\Z$, has a $\Z$-basis $\{h_\lambda(X)e_\mu(Y):(\lambda|\mu)\in\P^2\}$ and its $n$th component has a $\Z$-basis $\{h_\lambda(X)e_\mu(Y):(\lambda|\mu)\in\P^2(n)\}$.


Fix a positive integer $m$. We define the quotient \[\Gamma^{(m)}=\langle \Lambda(X),\Lambda(Y)\rangle/I\] where $I$ is the ideal of $\langle \Lambda(X),\Lambda(Y)\rangle$ generated by $\sum_{i=0}^d(-1)^ih_i(X)e_{d-i}(Y)$ for every positive integer $d$ such that $m\nmid d$. For each $i,j\in \NN_0$ and $\alpha,\beta\in \C$, we denote
\begin{align*}
h_i&=h_i(X)+I,& e_j&=e_j(Y)+I,&
h_\alpha&=h_\alpha(X)+I,& e_\beta&=e_\beta(Y)+I.
\end{align*} We call $\Gamma^{(m)}$ the ring of $m'$-truncated symmetric functions in the variables $X\cup Y$.

Since $I$ is a homogeneous ideal of $\langle \Lambda(X),\Lambda(Y)\rangle $, the ring $\Gamma^{(m)}$ is graded with the $n$th component as \[(\Gamma^{(m)})^n=(\langle \Lambda(X),\Lambda(Y)\rangle ^n+I)/I\] consisting of the $\Z$-linear combinations of the $m'$-truncated symmetric functions $h_\alpha e_\beta$ such that $|\alpha|+|\beta|=n$, together with the zero element.

\begin{thm}[{\cite[\S4.1(10,11)]{Donkin}}]\label{T: Gamma ring} The ring $\Gamma^{(m)}$ is the polynomial ring in $\{h_i,e_{jm}:i,j\in\NN\}$ and it has a $\Z$-basis \[\B=\{h_\lambda e_{m\mu}:(\lambda|m\mu)\in\P^2_m\}.\] Furthermore, there is a graded involution $\psi:\Gamma^{(m)}\to \Gamma^{(m)}$ given by $\psi(h_i)=e_i$.
\end{thm}

\subsection{Signed Young permutation module}\label{SS: 2.3} The general theory of the representation theory of the symmetric group can be found, for example, in \cite{GJAK}.

Let $m,n\in\NN_0$. We denote $\tsym{n}{m}$ for the symmetric group acting on the set $\{i+m:i\in\set{n}\}$ consisting of permutations of the forms $\tau^{+m}$ where $\tau\in\sym{n}$ such that, for all $i\in\set{n}$, \[\tau^{+m}(i+m)=\tau(i)+m.\] Furthermore, we identify $\sym{m}\times \sym{n}$ with the subgroup $\sym{m}\tsym{n}{m}$ of $\sym{m+n}$. For a given composition $\alpha=(\alpha_1,\ldots,\alpha_r)$ of $n$, we denote the Young subgroup of $\sym{n}$ with respect to $\alpha$ by \[\sym{\alpha}=\sym{\alpha_1}\times\cdots\times\sym{\alpha_r}.\]

For each $n\in\NN_0$, the set $\P^2_p(n)$ is partially ordered by the dominance order $\unrhd$ where $(\lambda|p\mu)\unrhd (\delta|p\xi)$ if
\begin{enumerate}[(a)]
  \item $\displaystyle{\sum^\ell_{i=1}\lambda_i\geq \sum^\ell_{i=1}\delta_i}$, and
  \item $\displaystyle{|\lambda|+p\sum^\ell_{i=1}\mu_i\geq |\delta|+p\sum^\ell_{i=1}\xi_i}$,
\end{enumerate} for all $\ell\in\NN$, where $\lambda_i$ is treated as $0$ if $i>\ell(\lambda)$ and similarly for $\mu_i$, $\delta_i$ and $\xi_i$. In the case when $(\lambda|p\mu)\unrhd (\delta|p\xi)$ and $(\lambda|p\mu)\neq (\delta|p\xi)$, we write $(\lambda|p\mu)\rhd(\delta|p\xi)$.

Recall that $k$ is a field of odd characteristic $p$. In this article, we denote $\otimes$ and $\boxtimes$ for the tensor product and outer tensor product of modules over $k$.

We write $k(n)$ and $\sgn(n)$ for the trivial and signature representations for $k\sym{n}$. For a composition $\alpha\in\C(n)$, $k(\alpha)$ and $\sgn(\alpha)$ are the restrictions of $k(n)$ and $\sgn(n)$ respectively to the Young subgroup $\sym{\alpha}$. Let $(\alpha|\beta)\in\C^2(n)$. We define the signed Young permutation $k\sym{n}$-module $M(\alpha|\beta)$ with respect to $(\alpha|\beta)$ as \[M(\alpha|\beta)=\Ind^{\sym{n}}_{\sym{\alpha}\times\sym{\beta}}(k(\alpha)\boxtimes \sgn(\beta)).\] Following the work of Donkin \cite[\S2.3]{Donkin}, we know that all non-isomorphic indecomposable summands of the signed Young permutation $k\sym{n}$-modules are parametrised by the set $\P^2_p(n)$. They are called the signed Young $k\sym{n}$-modules and denoted as $Y(\lambda|p\mu)$ such that, for each $(\lambda|p\mu)\in\P^2_p(n)$, $Y(\lambda|p\mu)$ is a summand of $M(\lambda|p\mu)$ with multiplicity one and the remaining summands of $M(\lambda|p\mu)$ are isomorphic to $Y(\delta|p\theta)$ for some $(\delta|p\theta)\in\P^2_p(n)$ and $(\delta|p\theta)\rhd(\lambda|p\mu)$.

Let $G$ be a group and $\R(kG)$ be the Green ring (or representation ring) of $kG$ consisting of formal linear combinations of the isomorphism classes $[V]$ of finite dimensional indecomposable $kG$-modules $V$ over $\Z$ such that $[V]=[W]$ if and only if $V\cong W$, $[V]+[W]=[V\oplus W]$ and the (inner) product given by $[V]\cdot[W]=[V\otimes W]$. The contragradient dual induces an automorphism on $\R(kG)$ since $(V\oplus W)^*\cong V^*\oplus W^*$ and $(V\otimes W)^*\cong V^*\otimes W^*$.

For each $n\in \NN_0$, we have the Green ring $\R(k\sym{n})$ where, by convention, $\R(k\sym{0})\cong \Z$ has a $\Z$-basis consisting only the trivial $k\sym{0}$-module. Let $V$ and $W$ be $k\sym{m}$- and $k\sym{n}$-modules respectively. We obtain a $k\sym{m+n}$-module given by $\Ind^{\sym{m+n}}_{\sym{m}\times \sym{n}} (V\boxtimes W)$. It is easily checked that $\Ind^{\sym{m+n}}_{\sym{m}\times \sym{n}} (V\boxtimes W)\cong \Ind^{\sym{m+n}}_{\sym{n}\times \sym{m}}(W\boxtimes V)$.  Let $\R_k(\sym{})=\bigoplus_{n\in\NN_0}\R(k\sym{n})$. We define an outer product on $\R_k(\sym{})$ by \[[V]\times [W]=[\Ind^{\sym{m+n}}_{\sym{m}\times \sym{n}} (V\boxtimes W)],\] and hence $[V]\times [W]=[W]\times [V]$. Then $\R_k(\sym{})$ forms a ring under the addition and outer product.

Since $[M(\alpha|\beta)]=[M(\lambda|\mu)]$ when $\alpha,\beta$ have types $\lambda,\mu$ respectively and $[M(\alpha|\beta)]\times [M(\gamma|\delta)]=[M(\alpha\cont\gamma|\beta\cont\delta)]$, the subset $\Y(\sym{})$ of  $\R_k(\sym{})$ spanned by the isomorphism classes of indecomposable summands of signed Young permutation modules is a graded subring where \[\Y(\sym{})=\bigoplus_{n\in\NN_0} \Y(\sym{n})\] and $\Y(\sym{n})$ is spanned by $[M(\alpha|\beta)]$ one for each $(\alpha|\beta)\in\P^2(n)$. 


\begin{thm}[{\cite[2.3(6--8)]{Donkin}}]\label{T: symmetric ring} Both $\{[Y(\lambda|p\mu)]:(\lambda|p\mu)\in\P^2_p\}$ and \[\CC=\{[M(\lambda|p\mu)]:(\lambda|p\mu)\in\P^2_p\}\] are $\Z$-bases for $\Y(\sym{})$ and there is a graded involution $\varpi:\Y(\sym{})\to\Y(\sym{})$ defined by $\varpi([V])=[V\otimes\sgn(n)]$ for each $[V]\in\Y(\sym{n})$.
\end{thm}

\subsection{Listing module} We refer the reader to \cite{JGreen} for the classical theory of the Schur algebras.

Fix $r\in\NN$ and let $\Omega$ be a set. The symmetric group $\sym{r}$ acts on $\Omega^r$ by place permutation and hence it induces a diagonal action on the set $\Omega^r\times \Omega^r$. The set of orbits of the action of $\sym{r}$ on $\Omega^r\times \Omega^r$ is denoted by $\O(\Omega^r)$. Let $f,g,h\in \O(\Omega^r)$, $(\seq{p},\seq{q})\in f$ and \[\lambda^f_{g,h}=|\{\seq{j}\in\Omega^r:\text{$(\seq{p},\seq{j})\in g$ and $(\seq{j},\seq{q})\in h$}\}|.\] Let $k(\Omega^r)$ be the formal vector space over $k$ with a basis $\{\xi_f:f\in \O(\Omega^r)\}$ endowed with a multiplication defined as \[\xi_g\xi_h=\sum_{f\in\O(\Omega^r)}(\lambda^f_{g,h}\cdot 1_k) \xi_f.\] Notice that $S(n,r)=k(\set{n}^r)$ is the classical Schur algebra. We denote the algebra $k(\NN^r)$ by $\Sch(\infty,r)$. By convention, we set $\Sch(\infty,0)=k\xi_0\cong k$. Furthermore, we denote \[\Sch(\infty)=\bigoplus_{r\in\NN_0} \Sch(\infty,r).\]

The algebra $\Sch(\infty)$ is a bialgebra as follows. Let $(\seq{i},\seq{j})\in \NN^r\times \NN^r$ and $(\seq{p},\seq{q})\in \NN^s\times \NN^s$ and $f,g$ be the orbits containing them, respectively. We denote the orbit containing $(\seq{i}\cont\seq{p},\seq{j}\cont\seq{q})\in\NN^{r+s}\times \NN^{r+s}$ as $f\cont g$. Define the comultiplication $\Delta:\Sch(\infty)\to \Sch(\infty)\otimes\Sch(\infty)$ and counit $\e:\Sch(\infty)\to k$ as
\begin{align*}
\Delta(\xi_f)&=\sum_{\substack{g\in\O(\NN^r),\ h\in\O(\NN^s)\\ f=g\cont h}} \xi_g\otimes \xi_h,& \e(\xi_f)&=\left \{\begin{array}{ll} 1&f=0,\\ 0&\text{otherwise.}\end{array}\right .
\end{align*} If $V,W$ are $\Sch(\infty,r)$- and $\Sch(\infty,s)$-modules respectively, then $V\otimes W$ is an $\Sch(\infty,r+s)$-module. 

Consider the natural $\Sch(\infty,1)$-module $E$ where $E$ is the $k$-vector space with a basis $\{e_i:i\in\NN\}$ and the action defined by \[\xi_f e_j=\left \{\begin{array}{ll}e_i&\text{if $f=\{(i,j)\}$,}\\ 0&\text{otherwise,}\end{array}\right .\] for all $j\in \NN$ and $f\in\O(\NN)$. For each $r\in\NN$, let $S^rE$ and $\bigwedge\nolimits^rE$ be the symmetric and exterior powers of $E$ respectively. By convention, $S^0E\cong k\cong \bigwedge\nolimits^0E$. For $\alpha\in\C(n)$, we set
\begin{align*}
S^\alpha E&=S^{\alpha_1}E\otimes\cdots\otimes S^{\alpha_{\ell(\alpha)}}E,& \bigwedge\nolimits^\alpha E&=\bigwedge\nolimits^{\alpha_1}E\otimes\cdots\otimes \bigwedge\nolimits^{\alpha_{\ell(\alpha)}}E,
\end{align*} so that $S^\alpha E$ and $\bigwedge\nolimits^\alpha E$ are $S(\infty,n)$-modules. For $(\alpha|\beta)\in \C^2(n)$, we define the mixed power the $\Sch(\infty,n)$-module $K^{\alpha|\beta}E:=S^\alpha E\otimes \bigwedge\nolimits^\beta E$. Notice that $K^{\delta|\xi}E\cong K^{\alpha|\beta}E$ if $\alpha,\beta$ have types $\delta,\xi$ respectively. Following Donkin \cite[\S4]{Donkin}, the class of non-isomorphic indecomposable summands of mixed powers for $\Sch(\infty,n)$ are parametrised by the set $\P^2_p(n)$. They are called the listing modules denoted by $\List^{\lambda|p\mu}E$ such that, for each $(\lambda|p\mu)\in\P^2_p(n)$, $\List^{\lambda|p\mu}E$ is a summand of $K^{\lambda|p\mu}E$ with multiplicity one and any other summand of $K^{\lambda|p\mu}E$ are isomorphic to $\List^{\delta|p\theta}E$ for some $(\delta|p\theta)\in\P^2_p(n)$ and $(\delta|p\theta)\rhd(\lambda|p\mu)$. Furthermore, under the Schur functor $f$, we have $f(K^{\alpha|\beta}E)\cong M(\alpha|\beta)$ and $f(\List^{\lambda|p\mu}E)\cong Y(\lambda|p\mu)$.

Let $\L(\Sch)$ be the graded subring of the representation ring of $\Sch(\infty)$ generated by $[S^nE]$ and $[\bigwedge\nolimits^n E]$ for every $n\in\NN_0$ with the gradation \[\L(\Sch)=\bigoplus_{n\in\NN_0} \L(\Sch(\infty,n))\] where $\L(\Sch(\infty,n))$ is spanned by the mixed powers $[K^{\alpha|\beta}E]$ such that $(\alpha|\beta)\in\C^2(n)$.

\begin{thm}\label{T: Schur ring} Both $\{[\List^{\lambda|p\mu}E]:(\lambda|p\mu)\in\P^2_p\}$ and \[\D=\{[K^{\lambda|p\mu}E]:(\lambda|p\mu)\in\P^2_p\}\] are $\Z$-bases for $\L(\Sch)$. Furthermore, the ring $\L(\Sch)$ has a graded involution $\omega:\L(\Sch)\to\L(\Sch)$ defined by $\omega(S^nE)=\bigwedge\nolimits^nE$.
\end{thm}

The following categorification theorem of the rings $\Y(\sym{})$ and $\L(\Sch)$ is essential throughout this article.

\begin{thm}[{\cite[\S4.1(10,13)]{Donkin}}]\label{T: Donkin decategorification} We have a commutative diagram
\[\xymatrix{\L(\Sch)\ar[rr]^{f}&&\Y(\sym{})\\ &\Gamma^{(p)}\ar[ur]_{\phi}\ar[ul]^{\theta}&\\
\L(\Sch)\ar[rr]^{f\hspace{1cm}}\ar[uu]^(0.3)\omega&&\Y(\sym{})\ar[uu]^(0.3)\varpi\\ &\Gamma^{(p)}\ar[ur]_{\phi}\ar[ul]^{\theta}\ar[uu]^(0.3)\psi&}\] where the maps $f,\theta,\phi$ are isomorphisms of graded rings where, for each $(\alpha|\beta)\in\C^2$, $f([K^{\alpha|\beta}E])=[M(\alpha|\beta)]$, $\theta(h_\alpha e_\beta)=[K^{\alpha|\beta}E]$ and $\phi(h_\alpha e_\beta)=[M(\alpha|\beta)]$.
\end{thm}

We end this section with a remark.

\begin{rem} Let $\m$ be the Mullineux map on $p$-restricted partitions (see, for example, \cite[\S6]{BrKu}) and let $\lambda(0)$ be the $p$-restricted partition such that $\lambda=\lambda(0)+p\xi$ for some partition $\xi$ (here, $\lambda(0)+p\xi$ has the obvious meaning componentwise summation of partitions). Using Theorem \ref{T: Donkin decategorification} and \cite[Theorem 3.21]{DL}, for each $(\lambda|p\mu)\in\P^2_p$, since $\varpi([Y(\lambda|p\mu)])=[Y(\m(\lambda(0))+p\mu|\lambda-\lambda(0))]$, we have \[\omega([\List^{\lambda|p\mu}E])=[\List^{\m(\lambda(0))+p\mu|\lambda-\lambda(0)}E].\] Similarly, in $\Gamma^{(p)}$, we have \[\psi(l_{\lambda|p\mu})=l_{\m(\lambda(0))+p\mu|\lambda-\lambda(0)}\] where, for each $(\alpha|p\beta)\in\P^2_p$, $l_{\alpha|p\beta}$ is the element such that $\phi(l_{\alpha|p\beta})=[Y(\alpha|p\beta)]$ and $\theta(l_{\alpha|p\beta})=[\List^{\alpha|p\beta}E]$ (see \cite[\S4.1]{Donkin}).
\end{rem}

\section{The ring $\Gamma^{(m)}$}\label{S: m'truncation}

Throughout this section, we fix $m\in\NN$. We derive some properties about the graded ring $\Gamma^{(m)}$. Recall that $\Gamma^{(m)}$ is the quotient ring $\langle \Lambda(X),\Lambda(Y)\rangle/I$ where $X=\{x_i:i\in\NN\}$, $Y=\{y_j:j\in\NN\}$ and $I$ is the ideal generated by $\sum_{i=0}^d(-1)^ih_i(X)e_{d-i}(Y)$ for every $d\in\NN$ such that $m\nmid d$. Furthermore, $\Gamma^{(m)}$ is the polynomial ring on $\{h_i,e_{jm}:i,j\in\NN\}$ where $h_n=h_n(X)+I$ and $e_n=e_n(Y)+I$, and it has a $\Z$-basis \[\B=\{h_\lambda e_{m\mu}:(\lambda|m\mu)\in\P^2_m\}.\]

We begin with a lemma.

\begin{lem}\label{L: Lambda 2} \
\begin{enumerate}[(i)]
\item  For each $d\in\NN$, $e_d$ is a $\Z$-linear combination of $h_\lambda e_{sm}$ for some partitions $\lambda$ and $s\in\NN_0$ such that $|\lambda|+sm=d$.
  \item Let \[a_{ij}=\left \{\begin{array}{ll} h_{1-i+j}&\text{if $m\nmid j$,}\\ (-1)^{j+1}e_j&\text{if $j=s m$ and $i=1$,}\\ 1&\text{if $j=sm$ and $i=s m+1$,}\\ 0&\text{otherwise.}\end{array}\right .\] Then, for each $d\geq 1$, we have $e_d=\det(a_{ij})_{1\leq i,j\leq d}$.
  \item If we set $x_i=y_i$ for each $i\in\NN$ then $I=0$ and $\Gamma^{(m)}=\Lambda(X)$ where $h_\lambda=h_\lambda(X)$ and $e_\mu=e_\mu(X)$.
\end{enumerate}
\end{lem}
\begin{proof} The proof of part (i) is straightforward by using induction on $d$ and the relation $\sum^d_{i=1}(-1)^ih_ie_{d-i}=0$ when $m\nmid d$. Alternatively, it follows from part (ii).

For part (ii), let $A_d=(a_{ij})_{1\leq i,j\leq d}$. Notice that $A_1=(e_1)$ and hence $e_1=\det(A_1)$. Fix a $d\in\NN$. By induction, suppose that $e_j=\det(A_j)$ for all $1\leq j\leq d-1$. Let $w_d$ be the $((d-1)\times 1)$-column matrix with the first entry $(-1)^{d+1}e_d$ and 0 elsewhere, let $z_i$ be the $(i\times 1)$-column matrix with the entries are $h_d, h_{d-1},\ldots,h_{d-i+1}$ reading from the top and let $v_j$ be the $(1\times j)$-row matrix with the last entry is 1 and 0 elsewhere. If $d=\ell m$ then, expand along the last column, \[\det(A_d)=\det\begin{pmatrix}A_{d-1}& w_d\\ v_{d-1}&0\end{pmatrix}=(-1)^{\ell m+1}\cdot (-1)^{\ell m+1}e_{\ell m}=e_d.\]  Suppose that $m\nmid d$. In particular, $m\neq 1$ and thus $h_1=e_1$. Let $B_s=\begin{pmatrix} A_{d-s}&z_{d-s}\\ v_{d-s}&h_s\end{pmatrix}$. We claim that, for all $1\leq s\leq d-1$, $\det(B_s)=\sum^d_{i=s}(-1)^{i-s}h_ie_{d-i}$. When $s=d-1$, we have \[\det(B_{d-1})=\det\begin{pmatrix} A_1&z_1\\ v_1&h_{d-1}\end{pmatrix}=e_1h_{d-1}-h_d=e_1h_{d-1}-e_0h_d.\] Expand $\det(B_s)$ along the last row, by induction, we have \begin{align*}
\det(B_s)=h_s\det(A_{d-s})-\det(B_{s+1})&=h_se_{d-s}-\sum^d_{i=s+1}(-1)^{i-(s+1)}h_ie_{d-i}\\
&=\sum^d_{i=s}(-1)^{i-s}h_ie_{d-i}.
\end{align*} Notice that $A_d=B_1$. Therefore, we conclude that \[\det(A_d)=\det(B_1)=\sum^d_{i=1}(-1)^{i-1}h_ie_{d-i}=e_d.\]


For part (iii), if $x_i=y_i$ for each $i\in \NN$, since $\sum_{i=0}^d(-1)^ih_i(X)e_{d-i}(X)=0$ for all $d$ by Theorem \ref{T: symm function}(ii), we have $I=0$. So $h_i=h_i(X)$, $e_i=e_i(X)$ and $\Gamma^{(m)}=\Lambda(X)$.
\end{proof}

In view of Lemma~\ref{L: Lambda 2}(i), we introduce the following notation. For each $n\in\NN$, we define $d_{n,\mu}\in \Z$, with respect to the basis $\B$, so that \[e_n=\sum_{n=|\mu|+sm} d_{n,\mu}h_\mu e_{sm}\] where $s\in\NN_0$ in the above summation. Notice that if $n=\ell m$ for some $\ell$ then $d_{\ell m,\mu}=0$ for all $\mu\neq\varnothing$ and $d_{\ell m,\varnothing}=1$. 


\begin{cor}\label{C: d_n=d_n+1}  Let $\mu\in \P(n)$. For all $r\in\NN_0$, we have \[d_{n+rm,\mu}=d_{n,\mu}.\]
\end{cor}
\begin{proof} For each $(\alpha|m\beta)\in\P^2_m$, we denote the coefficient of $h_\alpha e_{m\beta}$ in $z\in\Gamma^{(m)}$ as $(z,h_\alpha e_{m\beta})$. Then $d_{n,\mu}=(e_n,h_\mu)$ and $d_{n+rm,\mu}=(e_{n+rm},h_\mu e_{rm})$. Let $A_d=(a_{ij})_{1\leq i,j\leq d}$ be defined as in Lemma \ref{L: Lambda 2}(ii) and, by which, we have both $e_n=\det(A_n)$ and $e_{n+rm}=\det(A_{n+rm})$. By virtue of the entries of $A_d$, an element of the form $h_\alpha e_{jm}$ only appears in $e_{jm}(A_d)^{(1,jm)}$ where $(A_d)^{(1,jm)}$ is the $(1,jm)$-minor of the matrix $A_d$.

Let $C_d$ be the matrix obtained from $A_{d}$ by replacing all entries of the form $(-1)^{\ell +1}e_\ell$ by $0$. Expand along the $(rm)$th column of $A_n$, we have
\begin{align*}
(e_{n+rm},h_\mu e_{rm})&=(\det(A_{n+rm}),h_\mu e_{rm})\\
&=\left (e_{rm}(A_{n+rm})^{(1,rm)},h_\mu e_{rm}\right )\\
&=\left (e_{rm}\det\begin{pmatrix} J&B\\ \mathbf{0}&C_n\end{pmatrix},h_\mu e_{rm}\right )\\
&=(\det(C_n),h_\mu )
\end{align*} where $J$ is an upper unitriangular matrix and $B$ is some suitable matrix. By similar calculation, we have $(e_n,h_\mu)=(\det(C_n),h_\mu)$. So we conclude that $d_{n+rm,\mu}=d_{n,\mu}$.
\end{proof}

In view of Corollary~\ref{C: d_n=d_n+1}, we have the following notation.

\begin{nota}\label{N: d mu} Let $n\in\NN_0$. We denote \[\P(n;m)=\{\mu\in\P:\text{$n=|\mu|+sm$ for some $s\in \NN_0$}\}.\] For each partition $\mu$, we write $d_\mu=d_{a,\mu}$ for the common number for all $a=|\mu|+rm$ where $r\in \NN_0$ so that, for all $n\in\NN_0$, we have  \begin{equation}\label{Eq: 2}
  e_n=\sum_{n=|\mu|+sm} d_\mu h_\mu e_{sm}=\sum_{\mu\in\P(n;m)} d_\mu h_\mu e_{n-|\mu|}.
\end{equation}
\end{nota}

\section{The straightening rule}\label{S: straightening rule}

Recall the $\Z$-bases $\B$, $\CC$ and $\D$ of the rings $\Gamma^{(m)}$, $\Y(\sym{})$ and $\L(\Sch)$ denoted in Theorems \ref{T: Gamma ring}, \ref{T: symmetric ring} and \ref{T: Schur ring} respectively. Our main result in this section provides an explicit formula to write, for each $(\alpha|\beta)\in\C^2(n)$, the element $h_\alpha e_\beta$ in terms of $\B$. As a consequence, we have the explicit formulae to write $[M(\alpha|\beta)]$ and $[K^{\alpha|\beta}E]$ as $\Z$-linear combinations in terms of the bases $\CC$ and $\D$ respectively.

We begin with a list of notation.

\begin{nota} Fix a positive integer $m$, let $\lambda\in\P(n)$ and let $b=\lfloor\frac{n}{m}\rfloor$.
\begin{enumerate}[(i)]
\item Let $\C(\lambda)$ be the set of compositions of the form $\delta=(\delta_1,\ldots,\delta_{\ell(\lambda)})$ of type $\lambda$ and let $c_\lambda=|\C(\lambda)|$, i.e., \[c_\lambda=\frac{\ell(\lambda)!}{\prod_{i\in\NN}n_i!}\] where $n_i$ is the number of parts of $\lambda$ of size $i$.
\item Let $c_\lambda^{(m)}$ be the number of compositions $\delta=(\delta_1,\ldots,\delta_{\ell(\lambda)})\in\C(\lambda)$ such that \[\delta_1+\delta_2+\cdots+\delta_j\not\equiv 0\mod m\] for all $1\leq j\leq \ell(\lambda)$.
\item Let $\e_\lambda=(-1)^{|\lambda|-\ell(\lambda)}$.
\item For a finite sequence $A=(\delta^{(1)},\ldots,\delta^{(r)})$ of partitions, we write $c_A=\prod_{i=1}^rc_{\delta^{(i)}}$ and $\len{A}=r$. For example, $\len{(\varnothing)}=1$. Similarly, we write $c^{(m)}_A=\prod_{i=1}^rc^{(m)}_{\delta^{(i)}}$ and $\e_A=\prod_{i=1}^r\e_{\delta^{(i)}}$.
\item If $\delta,\xi$ are compositions such that $\delta\cont \xi$ has type $\lambda$, we write $\lambda=\delta\cup \xi$ (or $\lambda=\xi\cup\delta$). The notation $\lambda=\xi^{(1)}\cup\cdots\cup\xi^{(r)}$ has the obvious meaning given some compositions $\xi^{(1)},\ldots,\xi^{(r)}$.
\item Let $W(\lambda)$ be the set consisting of sequences $(\delta^{(1)},\ldots,\delta^{(r)},\xi)$ of partitions for some $r\in\NN_0$ such that $\xi\cup\delta^{(1)}\cup\cdots\cup\delta^{(r)}=\lambda$, $\delta^{(i)}\neq\varnothing$ and $m\mid |\delta^{(i)}|$ for all $1\leq i\leq r$, here $r$ may be zero.
\item For each $1\leq i\leq b$, let $W_i(\lambda)=\{A\in W(\lambda):\len{A}=i+1\}$.
\item For each $1\leq i\leq b$, let $P_i$ be the subset of $\C(\lambda)$ consisting of all compositions $(\mu_1,\ldots,\mu_{\ell(\lambda)})$ such that $\mu_1+\mu_2+\cdots+\mu_r=mi$ for some $1\leq r\leq \ell(\lambda)$.
\item Let $\xi$ be a partition and $\beta$ be a composition. Suppose that $s=\ell(\beta)$. Recall Notation \ref{N: d mu}. We denote by $V(\beta;(\xi,m\mu))$ the set consisting of the $s$-tuples \[(\xi^{(1)},\ldots,\xi^{(s)})\in \P(\beta_1;m)\times\cdots\times\P(\beta_{s};m)\] such that $\xi=\xi^{(1)}\cup\cdots\cup\xi^{(s)}$ and $m\mu=(\beta_1-|\xi^{(1)}|)\cup\cdots\cup (\beta_s-|\xi^{(s)}|)$. Furthermore, we write \[c^{(m)}_{\beta;(\xi,m\mu)}=\sum_{A\in V(\beta;(\xi,m\mu))} \e_A c_A^{(m)}.\]
\end{enumerate} By convention, $c_\varnothing=1$, $c_\varnothing^{(m)}=1$, $\e_\varnothing=1$, $W(\varnothing)=\{(\varnothing)\}$ and $c^{(m)}_{\beta;(\xi,m\mu)}=0$ if $V(\beta;(\xi,m\mu))=\emptyset$.
\end{nota}

For example, let $\lambda=(3,2,1,1)$ and $m=3$. Then the compositions $(\mu_1,\mu_2,\mu_3,\mu_4)\in \C(\lambda)$ with the property that $3\nmid \sum_{i=1}^j\mu_i$ for all $1\leq j\leq 4$ are
\begin{align*}
  &(1,1,3,2),&&(1,3,1,2),&&(1,1,2,3).
\end{align*} Thus $c^{(3)}_{(3,2,1,1)}=3$. Furthermore, \[W(\lambda)=\{(\lambda),((3),(2,1,1)),((2,1),(3,1)),((3,2,1),(1)),((3),(2,1),(1)),((2,1),(3),(1))\}.\]

For another example, if $\lambda\neq\varnothing$ and $m\mid |\lambda|$, then $c_\lambda^{(m)}=0$. Also, \[W((m))=\{((m)),((m),\varnothing)\}.\]

We may now state our main theorem.

\begin{thm}\label{T: straight sym m} Fix a positive integer $m$. For each $(\alpha|\beta)\in\C^2(n)$,  \[h_\alpha e_\beta=\sum c^{(m)}_{\beta;(\xi,m\mu)} h_\lambda e_{m\mu}\] where the sum runs over all $(\lambda|m\mu)\in\P^2_m(n)$ such that $\lambda=\alpha\cup \xi$ for some partition $\xi$. 
\end{thm}

We illustrate the theorem with an example.

\begin{eg}\label{Ex: 1} Let $m=3$. Then \[e_2=\sum_{2=|\lambda|+3s}\e_\lambda c^{(3)}_\lambda h_\lambda e_{3s}=-h_{(2)}+h_{(1,1)}\] since $c_{(2)}^{(3)}=1$, $\e_{(2)}=-1$, $c_{(1,1)}^{(3)}=1$ and $\e_{(1,1)}=1$. Similarly, we have \[e_4=-h_{(4)}+h_{(3,1)}+h_{(2,2)}-h_{(2,1,1)}+h_{(1)}e_{(3)}.\] So, for example, we get
\begin{align*}
h_{(5)}e_{(4,3,2)}&=h_{(5)}\left (-h_{(4)}+h_{(3,1)}+h_{(2,2)}-h_{(2,1,1)}+h_{(1)}e_{(3)}\right )e_{(3)}\left (-h_{(2)}+h_{(1,1)}\right )\\
&=h_{5,4,2}e_3-h_{5,4,1,1}e_3-h_{5,3,2,1}e_3+h_{5,3,1,1,1}e_3-h_{5,2,2,2}e_3+2h_{5,2,2,1,1}e_3\\
&\ \ \ -h_{5,2,1,1,1,1}e_3-h_{5,2,1}e_{3,3}+h_{5,1,1,1}e_{3,3},
\end{align*} here, after the last equality, we have deliberately removed the parentheses.
\end{eg}

To prove Theorem \ref{T: straight sym m}, we need the following two lemmas.

\begin{lem}\label{L: d lambda} For each partition $\lambda$, \[d_\lambda=-\e_\lambda\left (\sum_{A\in W(\lambda)} (-1)^{\len{A}}c_A\right ).\]
\end{lem}
\begin{proof} Let $b=\lfloor \frac{n}{m}\rfloor$. Recall that the elementary symmetric functions can be written in terms of the complete symmetric functions as \[e_n(X)=\sum_{\lambda\in\P(n)} \e_\lambda c_\lambda h_\lambda(X)\] with respect to the set of variables $X=\{x_i:i\in\NN\}$ (see \cite[Chapter I.2 Example 20]{Mac}). Similarly, $e_n(Y)=\sum_{\lambda\in\P(n)} \e_\lambda c_\lambda h_\lambda(Y)$. Let $x_i=y_i$ for each $i\in \NN$. Using Lemma \ref{L: Lambda 2}(iii) and Equation \ref{Eq: 2}, we have
\begin{equation}\label{Eq: 3}
\sum_{\mu\in\P(n;m)} d_\mu h_\mu(X) \left (\sum_{\delta\in\P(n-|\mu|)} \e_\delta c_\delta h_\delta(X)\right )=\sum_{\lambda\in\P(n)} \e_\lambda c_\lambda h_\lambda(X).
\end{equation} Fix a partition $\lambda\in\P(n)$. If $\lambda=\mu\cup\delta$ with $\mu,\delta\in\P$ then $\mu$ uniquely determines $\delta$ and vice versa. In Equation \ref{Eq: 3}, $m\mid |\delta|$ (since $\mu\in P(n;m)$). So the equation translates to \[\sum_{\lambda\in\P(n)}\left (\sum^b_{i=0}\sum_{\substack{\delta\in\P(mi),\\ \delta\cup \mu=\lambda}} \e_\delta d_\mu c_\delta\right )h_\lambda(X)=\sum_{\lambda\in\P(n)} \e_\lambda c_\lambda h_\lambda(X).\] By Theorem~\ref{T: symm function}(i), comparing the coefficients of $h_\lambda(X)$, we have
\begin{equation}\label{Eq: 1}
\e_\lambda c_\lambda=\sum^b_{i=0}\sum_{\substack{\delta\in\P(mi),\\ \delta\cup \mu=\lambda}} \e_\delta d_\mu c_\delta.
\end{equation}  We prove the equality in the statement by using induction on the size of $\lambda$. If $\lambda=\varnothing$ then $d_\varnothing=1$ and $-\e_\varnothing(-1)^{\len{(\varnothing)}}c_\varnothing=1$. Using Equation \ref{Eq: 1} and induction, we have
\begin{align*}
  \e_\lambda c_\lambda=d_\lambda+\sum^b_{i=1}\sum_{\substack{\delta\in\P(mi),\\ \delta\cup \mu=\lambda}} \e_\delta d_\mu c_\delta&=d_\lambda-\sum^b_{i=1}\sum_{\substack{\delta\in\P(mi),\\ \delta\cup \mu=\lambda}} \e_\delta \e_\mu c_\delta\left (\sum_{B\in W(\mu)} (-1)^{\len{B}}c_B\right )\\
  &=d_\lambda+\e_\lambda \sum^b_{i=1}\sum_{\substack{\delta\in\P(mi),\\ \delta\cup \mu=\lambda}} \left (\sum_{B\in W(\mu)} (-1)^{\len{B}+1}c_\delta c_B\right )
\end{align*} Notice that $W(\lambda)\backslash\{(\lambda)\}$ consists of sequences of partitions, one for each $i\in\set{b}$, $(\delta,\delta^{(1)},\ldots,\delta^{(j)},\xi)$ such that $\delta\cup \mu=\lambda$, $|\delta|=im$ and $(\delta^{(1)},\ldots,\delta^{(j)},\xi)\in W(\mu)$. So we get
\[d_\lambda=\e_\lambda\left (c_\lambda-\sum_{A\in W(\lambda)\backslash \{(\lambda)\}} (-1)^{\len{A}} c_A\right )=-\e_\lambda\left (\sum_{A\in W(\lambda)} (-1)^{\len{A}}c_A\right ).\] The proof is now complete.
\end{proof}

\begin{lem}\label{L: P and cA} Let $\lambda\in\P(n)$ and $b=\lfloor\frac{n}{m}\rfloor$.
\begin{enumerate}
  \item [(i)] For each $j\in\set{b}$, we have \[\sum_{1\leq i_1<\cdots<i_j\leq b} |P_{i_1}\cap \cdots \cap P_{i_j}|=\sum_{A\in W_j(\lambda)} c_A.\]
  \item [(ii)] We have $d_\lambda=\e_\lambda c^{(m)}_\lambda$.
\end{enumerate}
\end{lem}
\begin{proof}  For the proof of part (i), we denote the disjoint union of sets $S_1,\ldots,S_a$ as $\bigsqcup_{i=1}^a S_i$. Let \[\psi:\bigsqcup_{1\leq i_1<\cdots<i_j\leq b} P_{i_1}\cap \cdots \cap P_{i_j}\to \bigsqcup_{(\delta^{(1)},\ldots,\delta^{(j)},\xi)\in W_j(\lambda)} \C(\delta^{(1)})\times \cdots\times \C(\delta^{(j)})\times \C(\xi)\] be defined as follows. For each $\eta\in P_{i_1}\cap \cdots \cap P_{i_j}$, we have positive integers $1\leq s_1<\cdots<s_j\leq \ell(\lambda)$ such that $\sum^{s_r}_{t=1}\eta_t=mi_r$ for each $1\leq r\leq j$. Therefore, for each $1\leq t\leq j+1$, \[\eta^{(t)}=(\eta_{s_{t-1}+1},\eta_{s_{t-1}+2},\ldots,\eta_{s_t})\in \C(\delta^{(t)})\] for some partition $\delta^{(t)}$ of size $m(i_t-i_{t-1})$ where $i_0=0=s_0$ and $s_{j+1}=\ell(\lambda)$. Let $\xi=\delta^{(j+1)}$. Then \[\psi(\eta)=\prod_{t=1}^{j+1}\eta^{(t)}\in \C(\delta^{(1)})\times \cdots\times \C(\delta^{(j)})\times \C(\xi)\] since $(\delta^{(1)},\ldots,\delta^{(j)},\xi)\in W_j(\lambda)$. The map $\psi$ is invertible with the inverse $\phi$ given as follows. For each $(\delta^{(1)},\ldots,\delta^{(j)},\xi)\in W_j(\lambda)$ and \[h=(\eta^{(1)},\ldots,\eta^{(j)},\eta^{(j+1)})\in \C(\delta^{(1)})\times \cdots\times \C(\delta^{(j)})\times \C(\xi),\] we define $\phi(h)=\eta^{(1)}\cont\cdots\cont\eta^{(j)}\cont\eta^{(j+1)}\in P_{i_1}\cap \cdots \cap P_{i_j}$ where $mi_r=\sum_{t=1}^r|\eta^{(t)}|$ for $1\leq r\leq j$. The bijection $\psi$ shows that \[\sum_{1\leq i_1<\cdots<i_j\leq b} |P_{i_1}\cap \cdots \cap P_{i_j}|=\sum_{(\delta^{(1)},\ldots,\delta^{(j)},\xi)\in W_j(\lambda)} |\C(\delta^{(1)})\times \cdots\times \C(\delta^{(j)})\times \C(\xi)|=\sum_{A\in W_j(\lambda)} c_A.\] The proof for part (i) is now complete.

For part (ii), by the inclusion-exclusion principle and Lemmas~\ref{L: d lambda} and part (i), we have
\begin{align*}
c_\lambda^{(m)}=\left |\bigcap_{i=1}^b (\C(\lambda)\backslash P_i)\right |&=|\C(\lambda)|-\sum_{j=1}^b (-1)^{j+1}\left ( \sum_{1\leq i_1<\cdots<i_j\leq b} |P_{i_1}\cap \cdots \cap P_{i_j}|\right )\\
&=|\C(\lambda)|-\sum_{j=1}^b (-1)^{j+1}\left (\sum_{A\in W_j(\lambda)} c_A\right )\\
&=-\left (\sum_{A\in W(\lambda)} (-1)^{\len{A}}c_A\right )\\
&=\e_\lambda d_\lambda,
\end{align*} i.e., $d_\lambda=\e_\lambda c^{(m)}_\lambda$.
\end{proof}

We may now prove our main theorem.

\begin{proof}[Proof of Theorem \ref{T: straight sym m}] Let $s=\ell(\beta)$. By Equation \ref{Eq: 2} and Lemma \ref{L: P and cA}(ii), we have \[e_\beta=\prod_{i=1}^se_{\beta_i}=\prod_{i=1}^s\left (\sum_{\xi^{(i)}\in\P(\beta_i;m)}\e_{\xi^{(i)}}c_{\xi^{(i)}}^{(m)}h_{\xi^{(i)}} e_{\beta_i-|\xi^{(i)}|}\right ).\] Notice that $\prod_{i=1}^sh_{\xi^{(i)}}e_{\beta_i-|\xi^{(i)}|}=h_\xi e_{m\mu}$ where $\xi=\xi^{(1)}\cup\cdots\cup\xi^{(s)}$, $m\mu=(\beta_1-|\xi^{(1)}|)\cup\cdots\cup (\beta_s-|\xi^{(s)}|)$ and hence $(\xi|m\mu)\in \P^2_m(|\beta|)$ since $\xi^{(i)}\in \P(\beta_i;m)$. So, for an arbitrary $(\xi|m\mu)\in \P^2_m(|\beta|)$, by Theorem \ref{T: Gamma ring}, the coefficient of $h_\xi e_{m\mu}$ in $e_\beta$ is precisely $\sum \prod_{i=1}^s\e_{\xi^{(i)}}c_{\xi^{(i)}}^{(m)}$ where the sum is taken over all $s$-tuples $(\xi^{(1)},\ldots,\xi^{(s)})\in \P(\beta_1;m)\times\cdots\times\P(\beta_s;m)$ such that $\xi=\xi^{(1)}\cup\cdots\cup\xi^{(s)}$ and $m\mu=(\beta_1-|\xi^{(1)}|)\cup\cdots\cup (\beta_s-|\xi^{(s)}|)$, i.e., over the set $V(\beta;(\xi,\mu))$. So the coefficient of $h_\xi e_{m\mu}$ in $e_\beta$ is $c^{(m)}_{\beta;(\xi,m\mu)}$.  Therefore,
\begin{align*}
  h_\alpha e_\beta=h_\alpha\sum_{(\xi|m\mu)\in\P^2_m(|\beta|)}c^{(m)}_{\beta;(\xi,m\mu)} h_\xi e_{m\mu}&=\sum_{(\xi|m\mu)\in\P^2_m(|\beta|)}c^{(m)}_{\beta;(\xi,m\mu)} h_{\alpha\cup\xi} e_{m\mu}\\
  &=\sum_{\substack{(\lambda|m\mu)\in\P^2_m(n),\\ \lambda=\xi\cup\alpha,\ \xi\in\P}} c^{(m)}_{\beta;(\xi,m\mu)} h_\lambda e_{m\mu}.
\end{align*}
\end{proof}

Using Theorems \ref{T: Donkin decategorification} and \ref{T: straight sym m}, we obtain the following immediate corollaries.

\begin{cor}\label{C: decomposition of Young permutation} Let $(\alpha|\beta)\in\C^2(n)$. Then
\begin{enumerate}[(i)]
\item $[M(\alpha|\beta)]=\sum c^{(p)}_{\beta;(\xi,p\mu)} [M(\lambda|p\mu)]$, and
\item $[K^{\alpha|\beta}E]=\sum c^{(p)}_{\beta;(\xi,p\mu)} [K^{\lambda|p\mu}E]$
\end{enumerate} where both of the sums above run over all $(\lambda|p\mu)\in\P^2_p(n)$ such that $\lambda=\alpha\cup \xi$ for some partition $\xi$.
\end{cor}

We demonstrate Corollary \ref{C: decomposition of Young permutation} with an example. When $p=3$, following Example \ref{Ex: 1}, using the maps $\theta$ and $\phi$ as in Theorem~\ref{T: Donkin decategorification}, we have
\begin{align*}
  [M(\varnothing|(2))]&=-[M((2)|\varnothing)]+[M((1,1)|\varnothing)],\\
  [M(\varnothing|(4))]&=-[M((4)|\varnothing)]+[M((3,1)|\varnothing)]+[M((2,2)|\varnothing)]-[M((2,1,1)|\varnothing)]+[M((1)|(3))],\\
  [\bigwedge\nolimits^2E]&=-[S^2E]+[S^{(1,1)}E],\\
  [\bigwedge\nolimits^4E]&=-[S^4E]+[S^{(3,1)}E]+[S^{(2,2)}E]-[S^{(2,1,1)}E]+[E\otimes \bigwedge\nolimits^3E].
\end{align*}

Let $(M(\alpha|\beta):Y(\delta|p\xi))$ denote the multiplicity of $Y(\delta|p\xi)$ as a direct summand of $M(\alpha|\beta)$. Similarly, we have the multiplicity $(K^{\alpha|\beta}E:\List^{\delta|p\xi}E)$. The following corollary is clear.

\begin{cor} Let $(\alpha|\beta)\in\C^2(n)$ and $(\delta|p\xi)\in\P^2_p(n)$. Then
\begin{enumerate}[(i)]
\item $(M(\alpha|\beta):Y(\delta|p\xi))=\sum c^{(p)}_{\beta;(\xi,p\mu)} (M(\lambda|p\mu):Y(\delta|p\xi))$, and
\item $(K^{\alpha|\beta}E:\List^{\delta|p\xi}E)=\sum c^{(p)}_{\beta;(\xi,p\mu)} (K^{\lambda|p\mu}E:\List^{\delta|p\xi}E)$
\end{enumerate} where both of the sums above run over all $(\lambda|p\mu)\in\P^2_p(n)$ such that $\lambda=\alpha\cup \xi$ for some partition $\xi$.
\end{cor}

Recall that $Y(\lambda|p\mu)$ (respectively, $\List^{\lambda|p\mu}E$) is a summand of $M(\lambda|p\mu)$ (respectively, $K^{\lambda|p\mu}E$) with multiplicity one and any other summand $Y(\alpha|p\beta)$ (respectively, $\List^{\alpha|p\beta}E$) necessarily satisfies $(\alpha|p\beta)\rhd (\lambda|p\mu)$. Applying our result, we have a similar description when $(\lambda|p\mu)$ is replaced by a general pair of compositions $(\alpha|\beta)$.

\begin{cor}\label{C: Young distinguished summand} Let $(\alpha|\beta)\in\C^2(n)$ and $s=\ell(\beta)$. For each $1\leq i\leq s$, let $\beta_i=p\eta_i+r_i$ such that $0\leq r_i\leq p-1$ and let $r=\sum_{i=1}^{s}r_i$. Suppose that $\alpha\cont(1^r)$ and $(\eta_1,\ldots,\eta_{s})$ have types $\lambda$ and $\mu$ respectively.
\begin{enumerate}[(i)]
\item The signed Young permutation module $M(\alpha|\beta)$ has the summand $Y(\lambda|p\mu)$ such that $(M(\alpha|\beta):Y(\lambda|p\mu))=1$ and, if $Y(\delta|p\theta)$ is a summand of $M(\alpha|\beta)$, then $(\delta|p\theta)\unrhd (\lambda|p\mu)$.
\item The mixed power $K^{\alpha|\beta}E$ has the summand $\List^{\lambda|p\mu}E$ such that $(K^{\alpha|\beta}E:\List^{\lambda|p\mu}E)=1$ and, if $\List^{\delta|p\theta}E$ is a summand of $K^{\alpha|\beta}E$, then $(\delta|p\theta)\unrhd (\lambda|p\mu)$.
\end{enumerate}
\end{cor}
\begin{proof} We shall only prove part (i). Part (ii) can be proved in the similar fashion or follows from part (i) using Theorem~\ref{T: Donkin decategorification}. Since $M(\alpha|\beta)\cong M(\zeta|\nu)$ if $\zeta,\nu$ have types $\alpha,\beta$ respectively, we may assume that $\alpha,\beta$ are partitions so that $\lambda=\alpha\cont(1^r)$ and $\mu=(\eta_1,\ldots,\eta_s)$ in the statement. Notice that $r=|\beta|-p|\mu|$ and $(\lambda|p\mu)\in\P^2_p(n)$. It suffices to prove that
\begin{enumerate}
  \item [(a)] the coefficient of $[M(\lambda|p\mu)]$ in $[M(\alpha|\beta)]$ is 1, i.e., $c_{\beta;((1^r),p\mu)}^{(p)}=1$, and
  \item [(b)] if $(\delta|p\theta)\in\P^2_p(n)$ such that $M(\delta|p\theta)$ has the summand $Y(\lambda|p\mu)$ and $[M(\delta|p\theta)]$ appears in the sum of $[M(\alpha|\beta)]$ in Corollary \ref{C: decomposition of Young permutation}(i) then $(\delta|p\theta)=(\lambda|p\mu)$.
\end{enumerate}

For each $0\leq d\leq p-1$, we have $\e_{(1^d)} c_{(1^d)}^{(p)}=1$. It is easy to check that \[V(\beta;((1^r),p\mu))=\{((1^{r_1}),\ldots,(1^{r_s}))\}\] and hence $c_{\beta;((1^r),p\mu)}^{(p)}=1$. This proves statement (a).

Suppose that $(\delta|p\theta)\in\P^2_p(n)$ and $[M(\delta|p\theta)]$ appears in the sum of $[M(\alpha|\beta)]$ in Corollary \ref{C: decomposition of Young permutation}(i), i.e., $c^{(p)}_{\beta;(\xi,p\theta)}\neq 0$ for some partition $\xi$ such that $\delta=\alpha\cup\xi$, $\xi=\xi^{(1)}\cup\cdots\cup\xi^{(s)}$, $\xi^{(i)}\in \P(\beta_i;p)$ for each $1\leq i\leq s$ and $p\theta=(\beta_1-|\xi^{(1)}|)\cup\cdots\cup (\beta_s-|\xi^{(s)}|)$. Since $|\xi^{(i)}|\geq r_i$ for each $1\leq i\leq s$, we have $|\xi|\geq r$. So $\theta=(\eta_1-\epsilon_1)\cup\cdots\cup(\eta_s-\epsilon_s)$ for some nonnegative integers $\epsilon_1,\ldots,\epsilon_s$. Next, we shall show that $\theta_i\leq \eta_i$ for all $1\leq i\leq s$. Suppose, on the contrary, that $\theta_i>\eta_i$ for some $i$. For any $1\leq k\leq i-1$, we have $\eta_i<\theta_i\leq \theta_k=\eta_\ell-\epsilon_\ell$ for some $\ell$. In this case, $1\leq \ell\leq i-1$. Therefore the multisets (maybe empty) $\{\theta_k:1\leq k\leq i-1\}$ and $\{\eta_k-\epsilon_k:1\leq k\leq i-1\}$ are identical. So $\theta_i=\eta_j-\epsilon_j$ for some $j\geq i$ and hence $\eta_i<\theta_i\leq \eta_j$ for some $j\geq i$. This contradicts to the fact that $(\eta_1,\ldots,\eta_s)$ is a partition. So $\theta_i\leq \eta_i$ for all $1\leq i\leq s$. Suppose further that $Y(\lambda|p\mu)$ is a summand of $M(\delta|p\theta)$. We have $(\lambda|p\mu)\unrhd (\delta|p\theta)$. In particular, \[|\alpha|+r=|\lambda|\geq |\delta|=|\alpha|+|\xi|.\] Hence $|\xi|=r$, $|\lambda|=|\delta|$ and $\alpha\cont (1^r)=\lambda\unrhd \delta=\alpha\cup\xi$ where, here, $\unrhd$ is the usual dominance order on $\P(n)$. Therefore, $\xi=(1^r)$. Furthermore, $p|\theta|=p|\mu|$ and hence we conclude that $\theta=\mu$.
\end{proof}


In \cite[Proposition 7.1]{GLOW}, Giannelli, O'Donovan, Wildon and the author found the label of the signed Young module when $M(\alpha|\beta)$ is indecomposable (see \cite[Theorem 1.4]{GLOW})in the odd characteristic case, i.e., the pair of partitions $(\delta|p\theta)$ such that $M(\alpha|\beta)\cong Y(\delta|p\theta)$. Direct application of Corollary~\ref{C: Young distinguished summand} gives an alternative proof.

\begin{cor}[{\cite[Proposition 7.1]{GLOW}}] Let $a,b\in\NN_0$, $a\neq 0$ and $b=sp+r$ where $0\leq r\leq p-1$. Then $M(\varnothing|(b))\cong Y((1^r)|p(s))$ and, if $p\mid a+b$, $M((a)|(b))\cong Y((a,1^r)|p(s))$. Similarly, we have $K^{\varnothing|(b)}E\cong \List^{(1^r)|p(s)}E$ and, if $p\mid a+b$, $K^{(a)|(b)}E\cong \List^{(a,1^r)|p(s)}E$.
\end{cor}

\end{document}